\documentclass[12pt]{amsart}
    
    \usepackage[utf8]{inputenc}
    \usepackage[letterpaper, portrait, margin=1.25in]{geometry}
    \usepackage[justification=centering]{caption}
    \usepackage{longtable}

    \usepackage{amsmath}
    \usepackage{amssymb}
    \usepackage{amsthm}
    \usepackage{hyperref}

    \newtheorem{theorem}{Theorem}[section]
    \newtheorem{conj}[theorem]{Conjecture}
    \newtheorem{prop}[theorem]{Proposition}

    \theoremstyle{definition}
    \newtheorem{ex}[theorem]{Example}

    \binoppenalty=\maxdimen
    \relpenalty=\maxdimen

        \newcommand{\set}[1]{\left\{#1\right\}}
        \newcommand{\abs}[1]{{\left|#1\right|}}
        
            \newcommand{\seqq}[2][n]{{#2(#1)_{#1\geq 0}}}
            \newcommand{\seq}[2][n]{{#2(#1)}}
            \DeclareMathOperator{\rank}{rank}
            
            \DeclareMathOperator{\charp}{char}
            \renewcommand{\char}[1]{\ensuremath{\charp_{#1}}}

            \DeclareMathOperator{\degg}{deg}
            \renewcommand{\deg}{\ensuremath{\degg}}

            \renewcommand{\phi}{\varphi}
            \newcommand{\phic}{\overline{\phi}}

        \let\bset\mathbb

        \newcommand{\oeis}[1]{\cite[\href{http://oeis.org/#1}{#1}]{OEIS}}

\title[Complexity of powers of a constant-recursive sequence]{Complexity of powers of a \\ constant-recursive sequence}
\author{Eric Rowland}
\address{
	Department of Mathematics \\
	Hofstra University \\
	Hempstead, NY \\
	USA
}
\author{Jes\'us Sistos Barr\'on}
\address{
	Department of Mathematics \\
    University of Georgia \\
    Athens, GA \\
    USA
}

\thanks{This project was conducted as part of the 2023 NYC Discrete Math REU, funded by NSF grant DMS-2051026 and Jane Street.}

\begin{document}
    \begin{abstract}
        Constant-recursive sequences are those which satisfy a linear recurrence, so that later terms can be obtained as a linear combination of the previous ones. The rank of a constant-recursive sequence is the minimal number of previous terms required for such a recurrence. For a constant-recursive sequence $\seq{s}$, we study the sequence $\left(\rank \seq{s}^M\right)_{M\geq 1}$. We answer a question of Stinchcombe regarding the complexity of the powers of a constant-recursive sequence when the roots of the characteristic polynomial are not all distinct.
    \end{abstract}
    \maketitle
    \section{Introduction}
        It is well known that the set of constant-recursive sequences is closed under termwise addition and multiplication.
        In particular, if $s(n)_{n \geq 0}$ is constant-recursive and $M$ is a positive integer, then $(s(n)^M)_{n \geq 0}$ is also constant-recursive.
        For example, the Fibonacci sequence $F(n)_{n\geq 0}$ satisfies $F(n+2) = F(n+1) + F(n)$, and its first several powers satisfy
        \begin{align*}
            F(n + 3)^2 &= 2 F(n + 2)^2 + 2 F(n + 1)^2 - F(n)^2 \\
            F(n + 4)^3 &= 3 F(n + 3)^3 + 6 F(n + 2)^3 - 3 F(n + 1)^3 - F(n)^3 \\
            F(n + 5)^4 &= 5 F(n + 4)^4 + 15 F(n + 3)^4 - 15 F(n + 2)^4 - 5 F(n + 1)^4 + F(n)^4.
        \end{align*}
        In this paper, we are interested in the complexity of the $M$-th power of a constant-recursive sequence, measured by the size of its minimal recurrence.
        We refer to the size of the minimal recurrence satisfied by $s(n)_{n \geq 0}$ as its \emph{rank} and denote it by $\rank s$. In general, if $r = \rank s$, then the rank of $(s(n)^M)_{n \geq 0}$ is at most $\binom{M + r - 1}{M}$.
        Shannon~\cite{Shannon1974Exponentials} attributes this upper bound to Van der Poorten, in a paper that seems to have never been published, but it can be obtained from Van der Poorten's short survey~\cite{Poorten1974Recurrences}.

        As we show in Proposition~\ref{Fibonacci powers rank}, the ranks of the powers of the Fibonacci sequence attain this upper bound. However, powers of other sequences do not. For example, Ramanujan \cite{Hirschhorn1995Ramanujan} discovered three sequences $a(n)_{n \geq 0}, b(n)_{n \geq 0}, c(n)_{n \geq 0}$~\cite[\href{http://oeis.org/A051028}{A051028}, \href{http://oeis.org/A051029}{A051029}, \href{http://oeis.org/A051030}{A051030}]{OEIS}, defined by the recurrence $s(n + 3) = 82 s(n + 2) + 82 s(n + 1) - s(n)$ with various initial conditions, that satisfy the `amazing' identity $$a(n)^3 + b(n)^3 = c(n)^3 + (-1)^n.$$ For these sequences, the upper bound $\binom{M + 2}{M}$ yields the sequence $3, 6, 10, 15, \dots$. Yet, the sequence $\left(\rank \seq{a}^M\right)_{M\geq 1}$ is instead $3, 5, 7, 9, \dots$, and similarly for $\seq{b}$ and $\seq{c}$.

        In this paper, we study the question of determining the rank of the $M$-th power of a constant-recursive sequence. Our main result is Theorem \ref{thm:bounds_specific}, which states that if a sequence $\seqq{s}$ has rank $r$ but only $k$ different roots in its characteristic polynomial, then the rank of $(s(n)^M)_{n \geq 0}$ is at most $(r-k)\binom{M+k-1}{M-1} + \binom{M+k-1}{M}$. A related question has been studied by Shannon and Horadam~\cite{Shannon1971Bounds}, Shannon~\cite{Shannon1974Exponentials}, and Stinchcombe~\cite{Stinchcombe1998Exponentials}, namely that of obtaining explicit expressions for the recurrences satisfied by the power sequences. The important assumption in those papers is that the roots of the characteristic polynomial are all different. Our result answers a question posed by Stinchcombe for the case in which the roots are allowed to have multiplicities, and we refine the known bounds on the rank of the power sequences.

        In this paper, we also initiate the study of sequences $\left(\rank \seq{s}^M\right)_{M\geq 1}$ as objects of interest in their own right. We say that a sequence of integers is a {\itshape rank sequence} if it arises as $\left(\rank \seq{s}^M\right)_{M\geq 1}$ for some constant-recursive integer sequence $\seqq{s}$. For each $r \geq 1$, we conjecture that there are only finitely many rank sequences associated with sequences $\seqq{s}$ of rank $r$, and we are interested in enumerating them. For example, the only rank sequence associated with sequences of rank $1$ is the constant sequence $1, 1, 1, \dots$. We prove in Theorem \ref{thm:rank_seqs_for_rank2} that, for general sequences of rank $2$, the number of different rank sequences is $5$. We carry out extensive computational searches in Mathematica for rank sequences associated with sequences of ranks $3$, $4$, and $5$, obtaining the lower bounds $1, 5, 8, 41, 36$ on the number of different rank sequences for $r \in \{1, 2, 3, 4, 5\}$. The results are listed in Table \ref{tab:r2_r3} and in the appendix.

        In Section~\ref{SEC_PREM}, we go over the characteristic polynomial and the relationship between its roots and closed formulas for the terms of a sequence. Later, in Section~\ref{SEC_MULTIPLICITY}, we study the relationship between the multiplicity of the characteristic roots and the rank of product sequences. We use a related argument in Section~\ref{SEC_BOUNDS} to refine the bounds on the rank of power sequences. In Section~\ref{SEC_COMPUTING}, we characterize the different rank sequences for sequences of rank $2$, and list the results of the searches for higher ranks. Lastly, in Section~\ref{SEC_SYMBOLS}, we discuss the use of linear algebraic methods to study the question of describing all rank sequences.
    \section{Background on constant-recursive sequences}\label{SEC_PREM}
        This section introduces some of the notation used throughout the paper and reviews known results about constant-recursive sequences. We also determine the ranks of the powers of the Fibonacci sequence.
        
        First, a constant-recursive sequence is one which satisfies a {\itshape recurrence equation} of the form
        \begin{align*}
            s(n+r) = c_{r-1} s(n+r-1) + \dots + c_1 s(n+1) + c_0 s(n).
        \end{align*}
        for any sufficiently large $n$, where the coefficients $(c_{r-1}, \dots, c_0)$ are all constants and $c_0\neq 0$.
        Every constant-recursive sequence satisfies infinitely many recurrences, but there is only one of minimal size, that we define to be the {\itshape rank} of $\seq{s}$, and denote it by $\rank s$.
        \subsection{Characteristic and exponential polynomials}
            Every sequence has an associated {\itshape characteristic polynomial}, obtained from the list of coefficients of the minimal recurrence. For the tuple $(c_{r-1}, \dots, c_0)$, this polynomial is
            \begin{align}\label{eq:char_poly}
                \char{s}(x) = x^r - (c_{r-1}x^{r-1} + \dots + c_1 x + c_0).
            \end{align}
            Note that the rank of the sequence translates into the degree of the characteristic polynomial. More importantly, the roots of this polynomial can be used to determine a closed formula for every term in the sequence. 
            
            An exponential polynomial is any function of the form
            \begin{align*}
                E(x) = \sum_{j = 1}^{k}P_j(x)\rho_j^x,
            \end{align*}
            where each of the coefficients $P_j(x)$ is a polynomial, and the $\rho_j$ are nonzero complex numbers. One has the following characterization of constant-recursive sequences.
            \begin{theorem}\label{thm:exps_seqs}
                Every exponential polynomial
                \begin{align*}
                    s(n) = \sum_{j=1}^{k} P_j(n)\rho_j^n
                \end{align*}
                generates a constant-recursive sequence, in which the set $\set{\rho_1, \dots, \rho_k}$ is exactly the set of the roots of $\char{s}(x)$, and furthermore, $\deg P_j$ is exactly one less than the multiplicity of $\rho_j$.
            \end{theorem}
            For a given root $\rho$ of $\char{s}(x)$, we denote its multiplicity by $m_s(\rho)$ (if the context makes it clear, we may simply write $m(\rho)$). Because the first terms of a constant-recursive sequence do not necessarily satisfy a strict recurrence, there is a one-to-many relationship between exponential polynomials and constant-recursive sequences. The following theorem establishes the connection in the opposite sense. 
            
            \begin{theorem}\label{thm:seqs_exps}
                Let $\seqq{s}$ be constant-recursive. Then, there is a unique exponential polynomial such that
                \begin{align*}
                    s(n) = \sum_{j=1}^{k} P_j(n)\rho_j^n
                \end{align*}
                for sufficiently large $n$.
            \end{theorem}
        The closure of constant-recursive sequences under element-wise multiplication was shown algebraically by Van der Poorten~\cite{Poorten1974Recurrences}, and later by Stinchcombe~\cite{Stinchcombe1998Exponentials}, by looking at sequences through their exponential polynomials. For completeness we provide the short proof.
        \begin{theorem}
            Let $\seqq{s}, \seqq{t}$ be constant-recursive sequences. Then, the product sequence $\seq{st} = \seq{s}\seq{t}$ is also constant-recursive.
        \end{theorem}
        \begin{proof}
            Let $\seq{s} = \sum_{i} P_i(n)\rho_i^n$ and $\seq{t} = \sum_{j} Q_j(n)\varrho_j^n$ be the exponential polynomials for $\seqq{s}$ and $\seqq{t}$, respectively. Then,
            \begin{align*}
                \seq{s}\seq{t} = \sum_{i, j} P_i(n)Q_j(n) \cdot(\rho_i \varrho_j)^n .
            \end{align*}
            Which is, again, an exponential polynomial. By Theorem~\ref{thm:exps_seqs}, the product sequence $\seq{s}\seq{t}$ is constant-recursive.
        \end{proof}
         Moreover, Theorem~\ref{thm:exps_seqs} gives us information about the minimal recurrence satisfied by the product sequence, and the roots and multiplicities of its characteristic polynomial. We expand on this in the next section. For now, the expansion of the product of exponential polynomials allows us to settle the question about the rank of the powers of the Fibonacci sequence.
        \begin{prop}\label{Fibonacci powers rank}
            For each positive integer $M$, the rank of the sequence $\seqq{F^M}$ is exactly $M+1$. In other words, the rank sequence of the Fibonacci numbers is $2, 3, 4, 5, 6, \dots$.
        \end{prop}
        \begin{proof}
            The Fibonacci sequence has exponential polynomial $F(n) = \frac{1}{\sqrt{5}}\left(\phi^n - \phic^n\right)$, where $\phi = \frac{1 + \sqrt{5}}{2}$ and $\phic = \frac{1-\sqrt{5}}{2}.$ Expanding $\seq{F^M} =\prod_{i = 1}^{M} \frac{1}{\sqrt{5}}\left(\phi^n - \phic^n\right)$ yields an exponential polynomial with summands of the form $C\phi^j \phic^{M-j}$, where $C = \pm \frac{1}{\sqrt{5}^M}$, depending on the parity of $M-j$. Any two of these terms can be added together if and only if there is some $J$ for which $\phi^J = \phic^J$, but this is never true since $\abs{\phi}>1>\abs{\phic}$. Hence, the exponential polynomial does not simplify, and thus by Theorem~\ref{thm:exps_seqs}, the rank of $\seqq{F^M}$ is exactly $M+1$.
        \end{proof}

        Alternatively, it is enough to show that the determinant of the $M \times M$ matrix whose $(i, j)$ entry is $F(i + j - 2)^M$ is nonzero, since that implies there are no linear relations among the first $M$ shifts of $F(n)^M$. This was proven by Carlitz in \cite{Carlitz1966Determinants}. The sequence of determinants for each $M$ appears in the OEIS \oeis{A265944}.
    \section{The influence of multiplicity on the rank}\label{SEC_MULTIPLICITY}
        Here, for sequences $\seq{s}, \seq{t}$, we show how the multiplicity of the roots in the characteristic polynomials influence the rank of the product sequence $\seq{s}\seq{t}$. We know by Theorem \ref{thm:exps_seqs} that there is, at first, a relationship between the multiplicity of a root and the degree of the coefficient in the exponential polynomial.
        
        \begin{prop}\label{prop:multiplicity_rules}
            Let $\rho_1$ be a root of $\char{s}(x)$ with multiplicity $m_s(\rho_1)$, and $\rho_2$ be a root of $\char{t}(x)$ with multiplicity $m_t(\rho_2)$. If there are no other roots $\varrho_1, \varrho_2$ of $\char{s}(x), \char{t}(x)$, respectively, such that $\rho_1\rho_2 = \varrho_1\varrho_2$, then $\rho_1\rho_2$ is a root of $\char{st}(x)$ with a multiplicity equal to $m_{st}(\rho_1\rho_2) = m_s(\rho_1) + m_t(\rho_2) - 1$.
        \end{prop}
        \begin{proof}
            Let $P_1(x), Q_2(x)$ be the respective polynomial coefficients in the exponential polynomials. We have that $m_s({\rho_1}) = \deg P_1 + 1$ and $m_t({\rho_2}) = \deg Q_2 + 1$. Now, the term $({\rho_1}{\rho_2})^n$ does not arise from any other pair of roots, hence its coefficient will be exactly $(P_1Q_2)(x)$. Since $$\deg P_1 Q_2 = \deg P_1 + \deg Q_2,$$ then $m_{st}({\rho_1}{\rho_2}) = \deg P_1 Q_2 + 1 = m_s({\rho_1}) + m_t({\rho_2}) - 1$.
        \end{proof}
        For example, when two roots $\rho_1, \rho_2$ have multiplicity $1$, the resulting root $\rho_1\rho_2$ of the product sequence will also have multiplicity $1$, as occurred with the exponential polynomials of the powers of the Fibonacci sequence.

        When there are different pairs of roots with a shared product, we need to pay more attention to the particular coefficients, as sometimes they may cause the exponential polynomial to simplify, a phenomenon which we illustrate in Example \ref{ex:fibo_luc} later on. An analogous argument extrapolates this to products of more than $2$ sequences at a time.
        \begin{prop}\label{prop:multiplicity_rules_general}
            Let $\seq{s_1}, \dots, \seq{s_\ell}$ be constant-recursive sequences. For each $1\leq i \leq \ell$, let $\rho_i$ be a root of $\char{s_i} (x)$ with multiplicity $m_{s_i}(\rho_i)$. If there is no other set $\set{\rho_i'}$ (where each $\rho_i'$ is a root of $\char{s_i} (x)$) such that $\prod \rho_i = \prod \rho_i'$, then $\prod \rho_i$ is a root of $\char{\prod s_i} (x)$ with multiplicity $1 - \ell + \sum m_{s_i}(\rho_i).$
        \end{prop}
        For now, our focus is only on bounding the rank from above. The next theorem states, informally, that for purposes of maximizing the rank, it is better to have many different roots, rather than a few with large multiplicities.
        \begin{theorem}\label{thm:product_rank_general}
            Let $\seqq{s}, \seqq{t}$ be constant-recursive sequences of ranks $r_1$ and $r_2$, respectively. The rank of $\seqq{st}$ is at most $r_1r_2$, with equality achieved only if, for at least one of $\char{s}(x), \char{t}(x)$, all of the roots are different.
        \end{theorem}
        \begin{proof}
            Let $\rho_1, \dots, \rho_{k_1}$ be the distinct roots of $\char{s}(x)$, and let $\varrho_1, \dots, \varrho_{k_2}$ be the distinct roots of $\char{t}(x)$. Expanding the product of the two exponential polynomials yields another exponential polynomial with summands containing powers of the characteristic roots of the form $\rho_i \varrho_j$. The rank is maximized when these are all different, because then each of the expected multiplicities $m_s(\rho_i) + m_t(\varrho_j) - 1$ given by Proposition~\ref{prop:multiplicity_rules} contributes towards the total rank of the product sequence.

            The term $m_s(\rho_i)$ appears $k_2$ times, for each $i$. Similarly, the term $m_t(\varrho_j)$ appears $k_1$ times for each $j$. In total, there are $k_1k_2$ of these terms. Thus, the total sum is at most $k_2\sum \rho_i + k\sum \varrho_j - k_1k_2$, or equivalently, $r_1k_2 + k_1r_2 - k_1k_2 = r_1r_2 - (r_1 - k)(r_2 - k_2)$. Since $r_1\geq k_1$ and $r_2\geq k_2$, then the highest possible rank of the product sequence is bounded by $r_1r_2$. For equality to be possible, we need that $(r_1 - k_1)(r_2 - k_2)=0$, which occurs only when $r_1 = k_1$ or $r_2 = k_2$, meaning in each case that all of the respective roots are different.
        \end{proof}

        \begin{ex}
            Consider the sequence $\seqq{s}$ with terms $1, 1, 2, 1, 1, 1, 0\dots$ and the sequence $\seqq{t}$ with terms $1, 1, 1, 3, 17, 83, 345, \dots$; satisfying, respectively, the following minimal recurrences for all $n\geq 0$:
            \begin{align*}
                s(n+4) &= 2s(n+2) - s(n)\\
                t(n+3) &= 7t(n+2) - 16t(n+1) + 12t(n).
            \end{align*}
            The product sequence $\seqq{st}$ turns out to be $$1, 1, 2, 1, 3, 17, 0, 345, -1291, 4513, -30150, 48809, \dots,$$
            with its minimal recurrence needing only $10$ previous terms. Specifically, for $n\geq 0$,
            \begin{align*}
                \seq[n+10]{st} &= 30\seq[n+8]{st} -345\seq[n+6]{st} \\&+ 1900\seq[n+4]{st} - 5040\seq[n+2]{st} + 5184\seq{st}.
            \end{align*}
            The rank of $\seq{st}$ is bounded above by $12$, according to Theorem \ref{thm:product_rank_general}. However, the bound is not met because neither of the characteristic polynomials have all distinct roots. In fact, $\char{s}(x) = (x-1)^2(x+1)^2$ and $\char{t}(x) = (x-2)^2(x-3)$. The characteristic polynomial of the product sequence is $\char{st}(x) = (x - 2)^3(x-3)^2(x+2)^3(x+3)^2,$
            as expected by the multiplicity rules from Proposition \ref{prop:multiplicity_rules}.
        \end{ex}
            
    \section{Refining the bounds on the rank of powers of a sequence}\label{SEC_BOUNDS}
            In this section, we review the known bounds on the rank, and then use what we know about the multiplicity of different roots to obtain more specific bounds, which depend only on the additional parameter of the number of different roots. The proof uses a counting argument similar to the one from Theorem \ref{thm:product_rank_general}.
            
            In general, if $s$ has $k_1$ different roots, and $t$ has $k_2$ different roots, then the sequence $\seqq{st}$ has at most $k_1 k_2$ different roots, regardless of the multiplicities. This may not always be the case. In Example \ref{ex:fibo_luc}, we show that, when expanding the product of two exponential polynomials, there are two different phenomena which affect the final rank. First, if different pairs of roots have the same product, then the number of different summands will be lower than expected. For example, between the sets of roots $1, 3$  and $2, 6$ there are only $3$ different products, namely $2, 6$ and $18$, because $1\cdot 6 = 2\cdot 3$. Moreover, if the corresponding polynomial coefficients cancel each other out, then that product of roots ends up not being a root of the exponential polynomial, further reducing the rank of the resulting sequence.

            \begin{ex}\label{ex:fibo_luc}
                The Lucas sequence is defined with the same recurrence as the Fibonacci numbers, but with initial terms $L(0) = 2, L(1) = 1$. Their exponential polynomials are given by
                \begin{align*}
                    F(n) &= \frac{1}{\sqrt{5}}\left(\phi^n - \phic^n\right),\\
                    L(n) &= \phi^n + \phic^n,
                \end{align*}
                with $\phi = \frac{1 + \sqrt{5}}{2}, \phic = \frac{1 - \sqrt{5}}{2}$.
                Expanding the corresponding product, we obtain the following exponential polynomials for $\seq{F^2}, \seq{L^2},$ and $\seq{F}\seq{L}$.
                \begin{align*}
                    \seq{F^2} &= \frac{1}{5}(\phi^2)^n - \frac{2}{\sqrt{5}}(\phi\phic)^n + \frac{1}{5}(\phic^2)^n,\\
                    \seq{F}\seq{L}&= \frac{1}{\sqrt{5}}\left((\phi^2)^n - (\phic^2)^n\right),\\
                    \seq{L^2} &= (\phi^2)^n + 2(\phi\phic)^n + (\phic^2)^n.
                \end{align*}
                First, notice how the exponential polynomials for both $F(n)^2$ and $L(n)^2$ have three summands and not four, since $\phi\phic = \phic\phi$ and thus two of the terms in each expansion can be added together. Furthermore, the exponential polynomial for $F(n)L(n)$ has only two summands, because the coefficient of $\phi\phic$ turns out to be $0$. From here, by Theorem~\ref{thm:exps_seqs}, we have that $F(n)^2$ and $L(n)^2$ have rank $3$, while the product sequence $F(n)L(n)$ only has rank $2$, failing to meet the bound of $4$ given by Theorem \ref{thm:product_rank_general}.
            \end{ex}
            When talking about the powers of a sequence (i.e.\ the products of a sequence with itself), the set of roots of the characteristic polynomial is always the same. Therefore, there are many sets of roots that yield equal products (by taking products of the same roots in a different order). Thus, we can come up with stronger upper bounds on the rank of each power.

            A first bound on the rank of the powers of a sequence, given also in \cite{Shannon1971Bounds} and \cite{Stinchcombe1998Exponentials}, can be seen quickly when all the roots are different. We include a short proof and then we expand the logic to the general case.
            
            \begin{theorem}\label{thm:bounds_initial}
                Let $\seqq{s}$ be a sequence of rank $r$ for which the roots of the characteristic polynomial are all different. Then, for any positive integer $M$, the rank of $\seq{s^M}$ is at most
                \begin{align*}
                    \rank s^M \leq \binom{M+r-1}{M}.
                \end{align*}
            \end{theorem}
            \begin{proof}
                The roots of $\char{s^M}(x)$ will be all the different possible products of $M$ roots of $\char{s}(x)$ (possibly with repetition). Every one of these products can be associated with a different multiset of $M$ roots, up to the order of the factors. There are $\binom{M+r-1}{M}$ such multisets. Thus, there are at most $\binom{M+r-1}{M}$ different roots of $\char{s^M}(x)$, with the bound not being met in the case of different multisets having the same product or with some of the polynomial coefficients being $0$. By Proposition~\ref{prop:multiplicity_rules_general}, each of these products has a multiplicity of $1$, from where the rank of the power sequence is at most equal to the binomial coefficient.
            \end{proof}

            In Theorem~\ref{thm:product_rank_general}, for two general sequences, the maximum possible rank of the product sequence is maximized when the roots of one of the characteristic polynomials are all different. A similar statement applies to the particular case of the powers of a sequence. The next theorem refines the known bounds, and it follows as a corollary that the bounds from Theorem~\ref{thm:bounds_initial} apply even when not all roots are distinct.

            \begin{theorem}\label{thm:bounds_specific}
                Let $\seqq{s}$ be a sequence of rank $r$ whose characteristic polynomial has $k$ different roots. Then, for any positive integer $M$, the rank of $\seq{s^M}$ satisfies
                \begin{align*}
                    \rank s^M
                        &\leq (r-k)\binom{M+k-1}{M-1} + \binom{M+k-1}{M}.
                \end{align*}
            \end{theorem}
            \begin{proof}
                As in the proof of Theorem~\ref{thm:bounds_initial}, and because we only care about the upper bound, we may assume that different multisets of the roots correspond to different products, and that no polynomial coefficient vanishes.

                We use a combinatorial argument. From Proposition~\ref{prop:multiplicity_rules_general}, the multiplicity of each root of $\char{s^M}(x)$ will be the sum of $M$ different multiplicities plus $1 - M$. The rank of $\seq{s^M}$ is the sum of the multiplicities of its different roots. Thus, we count instead the number of times that each root (and thus, the corresponding multiplicity) appears in all the multisets.

                Let $r_1, \dots, r_k$ be the different roots of $\char{s}(x)$, with respective multiplicities $m(r_1)$ up to $m(r_k)$. The root $r_1$ may appear a different number of times depending on the multiset, but always a number between $0$ an $M$. In fact, we can identify the cases in which each particular value occurs, as shown below.
                \begin{itemize}
                    \item $1$ time for each multiset of $\set{r_2, \dots, r_k}$ of size $M-1$.
                    \item $2$ times for each multiset of $\set{r_2, \dots, r_k}$ of size $M-2$.
                    
                    $\vdots$
                    \item $M$ times for each multiset of $\set{r_2, \dots, r_k}$ of size $0$.
                \end{itemize}
                From here, the number of times that each root appears across all multisets is the sum
                \begin{align*}
                    \sum_{j = 1}^{M} j\binom{(k-1) + (M - j - 1)}{(k - 1) - 1} &= \sum_{j = 1}^{M} j\binom{M + k - j - 2}{k-2}\\
                    &= \sum_{j = 1}^{M} j\binom{M + k - j - 2}{M-j}.
                \end{align*}

                Thus, the sum of the multiplicities of all possible products cannot be larger than
                \begin{align*}
                    \sum_{i = 1}^{M} \sum_{j = 1}^{M} {}&j\binom{M + k - j - 2}{M-j} \cdot m(r_i) + (1 - M)\binom{M+k-1}{M}\\
                    &= \sum_{i = 1}^{M} m(r_i) \sum_{j = 1}^{M} j\binom{M + k - j - 2}{M-j} + (1 - M)\binom{M+k-1}{M}\\
                    &= r \sum_{j = 1}^{M} j\binom{M + k - j - 2}{M-j} + (1 - M)\binom{M+k-1}{M}.
                \end{align*}
    
                For simplicity, we use the change of variables $j' = M-j, k' = k-2$. We also expand the term on the right to obtain the following simplified upper bound.
                \begin{align*}
                    r\sum_{j' = 0}^{M-1} (M-j')\binom{k' + j'}{j'} - M\binom{M + k' + 1}{M} + \binom{M + k' + 1}{M}
                \end{align*}
    
                It suffices to show that, for all $k$ and $r$, the identity $$r\sum_{j' = 0}^{M-1} (M-j')\binom{k' + j'}{j'} - M\binom{M + k' + 1}{M} = (r-k'-2)\binom{M+ k' + 1}{M-1}$$
                holds. The proof is routine and can be realized through induction, by using the well-known Pascal's identity $\binom{n}{k} + \binom{n}{k+1} = \binom{n+1}{k+1}$ and the identity $\sum_{j' = 0}^{M}\binom{k' + j'}{j'} = \binom{k' + M + 1}{M}$.
            \end{proof}

        \begin{ex}
            Consider the sequence $\seqq{s}$ with terms $1, 1, 2, 1, -8, -37, -110, \dots;$ which satisfies for all $n\geq 0$ the recurrence
            \begin{align*}
                s(n+4)=5s(n+3)-9s(n+2)+7s(n+1)-2s(n).
            \end{align*}
            The rank sequence for $\seq{s}$ is the sequence $4, 9, 16, 25, 36, \dots$, which falls short of the bounds from Theorem \ref{thm:bounds_initial} when $r = 4$. This is because $\char{s}(x) = (x-1)^3(x-2)$ has degree $4$, but only $k = 2$ different roots. However, this behavior follows the bounds from Theorem~\ref{thm:bounds_specific}. The lists of coefficients for the minimal recurrences of its first powers are given in Table~\ref{tab:recurrence_growth}.
            \begin{table}[h]
                \centering
                \begin{tabular}{|c|c|}
                     \hline
                     $M$ & Coefficients \\
                     \hline\hline
                     $1$ & $(5, -9, 7, -2)$\\
                     $2$ & $(15, -96, 346, -777, 1131, -1070, 636, -216, 32)$\\
                     $3$ & \begin{tabular}{c}
                           $(37, -615, 6121, -40957, 195855, -693853, 1861075, -3825918, 6057376,$\\
                            $-7371808, 6832560, -4733920, 2372992, -812544, 169984, -16384)$
                        \end{tabular}\\
                        \hline
                \end{tabular}
                \caption{Minimal recurrences for powers of $\seq{s}$.}
                \label{tab:recurrence_growth}
            \end{table}
        \end{ex}

            Note how, if all the roots of the characteristic polynomial are different (in other words, if $r = k$), then our bound simplifies to that given by Shannon in \cite{Shannon1974Exponentials}. In the rest of the cases, the bounds are less than or equal to those given by Theorem~\ref{thm:bounds_initial}.

    \section{Searching for all possible rank sequences}\label{SEC_COMPUTING}
        The main aim of this paper was to study the number of different rank sequences that could arise for constant-recursive sequences. It was shown that there were two mechanisms through which a sequence could fail to meet the refined rank bounds, yielding different rank sequences. In this section, we characterize most of the rank sequences for sequences of rank $2$, and we list the results of computational searches for sequences of higher ranks.
        
        Recall that in order not to meet the bounds from Theorem \ref{thm:bounds_specific}, it was necessary to have distinct multisets of roots with a same product. In addition, whenever the polynomial coefficients cancelled out, we arrived at further reductions in the resulting rank. Now we note that, while many sequences share the roots of their characteristic polynomials, the polynomial coefficients are more uniquely tied to each sequence, by Theorem~\ref{thm:seqs_exps}. Indeed, two sequences share an exponential polynomial if and only if they are eventually equal to each other.
        
        For this reason, the cases in which the polynomial coefficients cancel out occur sporadically among the set of all rank sequences. Thus, we refer to these as {\itshape particular} rank sequences, whereas all others which can be understood uniquely from the roots and their product relationships are denoted by {\itshape general} rank sequences.
    
        The bounds given in the previous section give an initial set of possible rank sequences, so long as it is possible to find a sequence for which the rank bounds are always sharp. That is, as long as there are sequences whose powers always attain the maximum possible rank. Proving their existence is straightforward, even when restricting ourselves to recurrences with only rational (or integer) coefficients. To find one, simply use the recurrence whose characteristic polynomial has prime numbers as its roots, with appropriate multiplicities to have a given rank $r$ and number of different roots $k$. By the fundamental theorem of arithmetic, different multisets of roots correspond to different roots of the power sequences, and thus the bounds will be attained.

        However, a direct computational search quickly reveals that these examples are far from encompassing all possibilities. In Tables~\ref{tab:r2_r3} and the table in the appendix, we list the different general rank sequences found for sequences with initial ranks $2, 3, 4$ and $5$, where the boldfaced sequences are those which arise from the bounds given in Theorem~\ref{thm:bounds_specific}. Finding particular rank sequences in this way is a more difficult problem, because we need to find, for each possible recurrence, a set of initial values that make the polynomial coefficients of the exponential polynomial cancel out at some point. A couple of examples are given below.

        \begin{ex}\label{ex:poly_cancel}
            Consider the sequence $\seqq{s}$ with terms $1, 1, -1, -1, 1, 1, \dots$; whose minimal recurrence is $s(n+2) = -s(n)$. The characteristic polynomial is $x^2 + 1$, with roots $\rho_1 = i, \rho_2 = -i$. Since $\rho_1^2 = \rho_2^2$, the general rank sequence for sequences with the same recurrence is $2, 2, 2, \dots$, because, by Theorem~\ref{thm:exps_seqs}, when expanding the exponential polynomials, the only distinct roots of $\char {s^M}(x)$ are $\rho_1^{M}$ and $\rho_1^{M-1}\rho_2$, each with multiplicity $1$. However, we can quickly observe that the sequence $\seqq{s^2}$ is $1, 1, 1, 1, 1,\dots$, and it has rank $1$, because its minimal recurrence is $s(n+1) = s(n)$. In fact, the rank sequence of $\seq{s}$ is the particular rank sequence $2, 1, 2, 1, 2,\dots$. The exponential polynomial of $\seq{s}$ is $\seq{s} = \left(\frac{1 - i}{2}\right)(i)^n +  \left(\frac{1 + i}{2}\right)(-i)^n$. Squaring this polynomial yields the exponential polynomial $1 \cdot (\rho_1\rho_2)^n = 1 \cdot (1)^n$ of only one summand, because the coefficient for $\rho_1^2 = \rho_2^2$ is $\left(\frac{1 - i}{2}\right)^2 + \left(\frac{1 + i}{2}\right)^2 = 0$.
        \end{ex}
        \begin{ex}
            We can extend the previous example to find sequences of higher rank with (eventual) coefficient cancellation in the rank sequence. Consider the sequence $\seqq{s}$ given by $\seq{s} = \left(\frac{1 - i}{2}\right)(i)^n +  \left(\frac{1 + i}{2}\right)(-i)^n + 1\cdot (2)^n$, namely $2, 3, 3, 7, 17, \dots$. By Theorem~\ref{thm:exps_seqs}, this is constant-recursive. Its characteristic polynomial is
            \begin{align*}
                \char{s}(x) = (x^2 + 1)(x-2) = x^3 - 2x^2 + x - 2,   
            \end{align*}
            and therefore its minimal recurrence is $$s(n+3) = 2s(n+2) - s(n+1) + 2s(n).$$ The general rank sequence for sequences with the same recurrence is $3, 5, 7, 9, 11, \dots$. Yet, the rank sequence for $\seqq{s}$ is $3, 4, 6, 7, 9, \dots$, because again the polynomial coefficients of $i^n$ and $(-i)^n$ cancel each other out when $n$ is even.
        \end{ex}

         Direct search among sequences with integer coefficients suggests that there are $8$ different general rank sequences for sequences of rank $3$, and at least $41$ for sequences of rank $4$. Moreover, the following conjecture seems to hold.
         \begin{conj} \label{conj:weak_rank_seqs}
             Every general rank sequence is eventually polynomial. Every particular rank sequence is eventually pseudo-polynomial.
         \end{conj}
        \begin{table}[h!]
            \centering
            \begin{tabular}{|c|l|c|}
                \hline
                 Coefficients &Rank Sequence & Eventual polynomial  \\
                \hline\hline
                 $0, 1$ &$2, 2, 2, 2, 2, 2, 2, 2,\dots$ & $2$ \\ 
                 $1, -1$ &$2, 3, 3, 3, 3, 3, 3, 3, \dots$ & $3$ \\ 
                 $2, -2$ &$2, 3, 4, 4, 4, 4, 4, 4, \dots$ & $4$ \\ 
                 $3, -3$ &$2, 3, 4, 5, 6, 6, 6, 6, \dots$ & $6$ \\ 
                 $1, 1$ &$\boldsymbol{2, 3, 4, 5, 6, 7, 8, 9, \dots}$ & $M+1$  \\ 
                \hline
                 $0, 0, 1$ &$3, 3, 3, 3, 3, 3, 3, 3, \dots$ & $3$  \\ 
                 $1, -1, 1$ &$3, 4, 4, 4, 4, 4, 4, 4, 4, \dots$ & $4$ \\ 
                 $2, -2, 1$ &$3, 5, 6, 6, 6, 6, 6, 6, 6, \dots$ & $6$ \\ 
                 $1, 1, -1$ &$\boldsymbol{3, 5, 7, 9, 11, 13, 15, 17, \dots}$ & $2M + 1$  \\ 
                 $1, 1, 2$ &$3, 6, 9, 12, 15, 18, 21, 24, \dots$ & $3M$  \\ 
                 $2, 0, -1$ &$3, 6, 10, 14, 18, 22, 26, 30, \dots$ & $4M - 2$ \\ 
                 $2, 0, -3$ &$3, 6, 10, 15, 21, 27, 33, 39, \dots$ & $6M - 9$ \\ 
                 $4, 11, -30$&$\boldsymbol{3, 6, 10, 15, 21, 28, 36, 45, \dots}$ & $\frac{M^2}{2} + \frac{3M}{2} + 1$ \\
                \hline
            \end{tabular}
            \caption{Rank sequences generated from sequences of starting ranks $2$ and $3$.}
            \label{tab:r2_r3}
        \end{table}
        
        Computing all the different rank sequences quickly becomes a computationally expensive task. When performing these searches, the sequences of ranks $4$ and $5$ were significantly constrained in the recurrences they could satisfy, and only for rank $2$ have we been able to prove that they are the only general possibilities, by generalizing the argument from Example~\ref{ex:poly_cancel}, as shown in the following theorem.

        \begin{theorem}\label{thm:rank_seqs_for_rank2}
            Let $\seqq{s}$ be a sequence of rank $2$ satisfying a recurrence with rational coefficients, for which the rank sequence is a general sequence (that is, without coefficient cancellation in the powers of the exponential polynomial). For any positive $M$, the rank of $\seqq{s^M}$ is either always equal to $M+1$, or it eventually becomes constant, equal to one of $\set{2, 3, 4, 6}.$ In total, there are $5$ different general rank sequences.
        \end{theorem}
        \begin{proof}
            For any sequence of rank $2$, the characteristic polynomial will have either a single root $\alpha$ with multiplicity $2$, or two different roots $\alpha, \beta$. The first case is straightforward. The characteristic polynomial of $\seqq{s^M}$ will have only one root $\alpha^M$ with multiplicity $2M + (1-M) = M+1$, thus the rank of $\seq{s^M}$ is always $M+1$.

            For the second case, the roots of the characteristic polynomial of the different powers $\seq{s^M}$ will be of the form $\alpha^k\beta^{M-k}$, each with multiplicity $1$. If at some point two of these are equal, that is, $\alpha^k\beta^{M-k} = \alpha^j\beta^{M-j}$ for some $M, k,$ and $j$ (with $k\geq j$), then we must also have that $\alpha^{k-j} = \beta^{k-j}$.

            If there is an equality like this, by the well-ordering principle there exists a least $K$ that satisfies the equality. From here, the rank of $\seqq{s^M}$ will be $M+1$ for $M < K$, but it will be exactly $K$ for $m\geq K$, as in this case, the different roots (all with multiplicity $1$) are $\beta^M, \alpha\beta^{M-1}, \dots, \alpha^{K-1}\beta^{M + 1 - K}$. Any multiset of $M$ roots has a product necessarily equal to one of these.

            The question, then, is to find the possible equalities between corresponding powers of the roots $\alpha, \beta$. The roots are complex numbers, so infinitely many equalities occur. In fact, for every $K$, the pair of roots of unity $\alpha = e^{i\pi/K}, \beta = e^{-i\pi/K}$, satisfies $\alpha^K = \beta^K$. However, only a few of these come from a characteristic polynomial with rational coefficients. In fact, the only possibilities for $K$ are $2, 3, 4$ and $6$, as we now show.

            Since the characteristic polynomial has only real coefficients, its two roots must be either real or complex conjugates. In any case, we must have $\abs{\alpha} = \abs{\beta}$. In the real case, the only way in which both roots are distinct is if $\alpha = -\beta$, for which the equality is necessarily $\alpha^2 = \beta^2$, and thus the rank stabilizes at $2$.

            In the case with complex roots, we have from Vieta's formulas that $\alpha\beta = \abs{\alpha}^2$ and $\alpha + \beta = 2\mathrm{Re}(\alpha)$ must be rational numbers. This occurs only if the square cosine of the argument of $\alpha$ is rational. On the other hand, for $\alpha^K$ to equal $\beta^K$, we must have the argument of $\alpha$ equal to a rational multiple of $\pi$.

            By the generalized Niven's theorem (see \cite{Nunn2021Nivens}), this only occurs when the argument is an integer multiple of $\frac{\pi}{4}$ or $\frac{\pi}{6}$. Each of these multiples will yield one of the rank sequences $2, 3, 4, 5, 6, 6, \dots; 2, 3, 4, 4, \dots; 2, 3, 3, \dots$, and $2, 2, 2, \dots$. In fact, in the first quadrant, these sequences arise, respectively, when $\alpha$ is $\frac{\pi}{6}, \frac{\pi}{4}, \frac{\pi}{3},$ and $\frac{\pi}{2}$, and $\beta = \overline{\alpha}$.
        \end{proof}
        This proof suggests that dealing with higher initial ranks may not be a straightforward process. There is no generalization of Niven's theorem to deal with the different possible identities that can arise when having more than two different roots. In fact, we don't yet know which product equalities can arise with three complex roots. 

    \section{A way forward: Rank sequences using symbolic roots}\label{SEC_SYMBOLS}
        In this section, we discuss an alternative way to visualize the different root products, and describe how a symbolic approach may be useful to understand the set of rank sequences. We have seen that, for most sequences $\seqq{s}$, the main parameter determining their rank sequence is only the minimal recurrence they satisfy, because the behavior of the exponential polynomial is given mainly by the interactions between the roots of the characteristic polynomial, with only a few examples causing further coefficient cancellation that reduces the expected rank.

        On top of that, not even the individual values of these roots have a direct influence on the rank sequence. Rather, it is the possible product equalities between multisets of them which constrain the growth of the rank. For this reason, using symbolic roots may offer a simpler way to characterize the behavior of all rank sequences.

        We do this as follows. For a given set of roots $R = \set{\rho_1, \dots, \rho_k}$ of some sequence $\seqq{s}$, take the natural imbedding of the different products between them into the free abelian group $\bset{Z}^k$ generated by $R$. For example, with 4 roots, the product $\rho_1^2\rho_3$ corresponds to the point $(2, 0, 1, 0)$ in $\bset{Z}^4$. Elements with negative coordinates arise from expressions in which some of the roots are in the denominator of a fraction.

        Then, the roots of $\seqq{s^M}$ are counted by those points of $\bset{Z}^k$ with nonnegative coordinates for which the sum of the coordinates is $M$. This approach has the disadvantage that it does not account for the multiplicity of the roots, but it shows another way to arrive to the bounds given in Theorem~\ref{thm:bounds_initial}.

        Any equality of products between roots can then be seen as a relation between the generators of the free group. For example, with three roots, an equality $\rho_1^2 = \rho_2\rho_3$ occurs if and only if $\frac{\rho_1^2}{\rho_2\rho_3}=1$, which can be understood as a vector $(2, -1, -1)$ whose coordinates add up to $0$. Note that the division of any two relations yields another relation (likewise, substracting two vectors whose coordinates add up to zero results in a vector satisfying the same constraint). Therefore, equalities between products are in bijection with some subgroup $N$ of $\bset{Z}^k$, and all the different products are encoded in the quotient group $\bset{Z}^k/N$. 
        
        The question becomes that of counting the number of equivalence classes represented by points whose coordinates have a fixed sum. First, given a set of generators for $N$ in row-vector form, the structure of the quotient group can be found by computing the Smith normal form of the matrix obtained from the row-vectors (see \cite[p.\ 343--345]{Hungerford2003Algebra}).

        \begin{ex}\label{ex:rank_5_lattice}
            Consider the sequence $\seqq{t}$ with initial terms $1, 2, -1, 1, 1, 2,\dots$ and satisfying the minimal recurrence
            \begin{align*}
                \seq[n+5]{t} = 2t(n+4) + t(n+3) - 2t(n+2) - t(n+1) - t(n).
            \end{align*}
            The rank sequence for $\seq{t}$ is $5, 15, 35, 67, 111, 167, 235, 315, \dots$ as listed in the table in the appendix, which is eventually equal to the polynomial $6M^2 - 10M + 11$. For sequences with $5$ different roots (without multiplicity), the sequence of bounds on the rank is described by the quartic polynomial $\binom{M+4}{M}$, but in this case the polynomial is only quadratic.
            
            The characteristic polynomial $\char{t}(x)$ has $5$ different roots, namely
            \begin{align*}
                \rho_1 &= -1 & \rho_2&\approx-0.192 - 0.548i & \rho_3&\approx -0.192 + 0.548i\\
                \rho_4 &\approx 1.692 - 0.318i & \rho_5&\approx1.692 - 0.318i, & &
            \end{align*}
            which satisfy the equalities $\rho_1^4 = \rho_2\rho_3\rho_4\rho_5$ and $\rho_1^2\rho_2\rho_4 = \rho_3^2\rho_5^2$, corresponding to the subgroup $N$ of $\bset{Z}^5$ generated by $(4, -1, -1, -1, -1)$ and $(2, 1, -2, 1, -2)$. The Smith normal form of the matrix $$\begin{pmatrix}
                4& -1& -1& -1& -1\\
                2& 1& -2& 1& -2
            \end{pmatrix}$$
            is given by
            $$\begin{pmatrix}
                1& 0& 0& 0& 0\\
                0& 3& 0& 0& 0
            \end{pmatrix},$$
            so that the quotient group $\bset{Z}^5/N$ is isomorphic to $\bset{Z}_1 \oplus \bset{Z}_3 \oplus \bset{Z}^3 = \bset{Z}_3 \oplus \bset{Z}^3$.
            Note how the degree of the eventual polynomial $6M^2 - 10M + 11$ is one less than the rank of the free part $\bset{Z}^3$.
        \end{ex}
        
        Heuristically, since the rank of the free part of the quotient group $\bset{Z}/N$ is reduced by $1$ with each generator for $N$, it seems like so is the degree of the conjectured polynomial that describes the rank sequence.
        This yields a partial refinement of Conjecture \ref{conj:weak_rank_seqs}.

        \begin{conj}
            Given a recurrence whose characteristic polynomial has $k$ distinct roots (without multiplicity), and $N$ is the subgroup of relations between them, its general rank sequence is eventually given by a polynomial with degree $d-1$, where $d$ is the rank of the free part of the quotient group $\bset{Z}^k/N$.
        \end{conj}

        As mentioned, however, this approach comes with new limitations. We have already mentioned that the multiplicities of the roots are not explicitly present in this lattice structure, even though they are an important component of the rank sequence. Moreover, while the symbolic framework ignores information that does not meaningfully influence the rank, it cannot answer the question of which relations can actually arise for sets of roots which are algebraic numbers. This matters because we usually only care about recurrences with rational or integer coefficients. We showed that for two roots, a rule of the form $\rho_1^7 = \rho_2^7$ is impossible, because otherwise we would have another general rank sequence for sequences of rank $2$. Another possibility may be that there are rules which cannot come on their own, but rather always accompanied by others.

    \newpage
    \section*{Appendix: Rank sequences for sequences of ranks 4 and 5}
    Here, we list some rank sequences for sequences of rank $4$ and $5$. Given the computational cost of computing the rank of higher powers, the data on these and sequences of higher ranks is limited.
    \scriptsize
    \begin{longtable}{|c|l|c|}
             \hline
             Coefficients &Rank Sequence & Eventual Polynomial  \\
            \hline
             $0,0,0,1$ &$4, 4, 4, 4, 4, 4, 4, 4, \dots$ & $4$  \\ 
             $1, -1, 1, -1$ &$4, 5, 5, 5, 5, 5, 5, 5, \dots$ & $5$ \\ 
             $0, 1, 0, -1$ &$4, 6, 6, 6, 6, 6, 6, 6, \dots$ & $6$ \\ 
             $0, 2, 0, -2$ &$4, 6, 8, 8, 8, 8, 8, 8, \dots$ & $8$ \\ 
             $2, 0, -4, 4$&$4, 8, 8, 8, 8, 8, 8, 8, \dots$ & $8$ \\ 
             $0, 3, 0, -3$&$4, 6, 8, 10, 12, 12, 12, 12, \dots$ & $12$ \\ 
             $1, -2, 1, -1$&$4, 8, 11, 12, 12, 12, 12, 12, \dots$ & $12$ \\ 
             $3, 0, -9, 9$&$4, 8, 12, 12, 12, 12, 12, 12, \dots$ & $12$ \\ 
             $2, -2, 4, -4$&$4, 9, 12, 12, 12, 12, 12, 12, \dots$ & $12$ \\ 
             $0, 1, 0, 1$ &$4, 6, 8, 10, 12, 14, 16, 18, \dots$ & $2M + 2$ \\ 
             $1, 0, 1, -1$ &$\boldsymbol{4, 7, 10, 13, 16, 19, 22, 25, \dots}$ & $3M+1$ \\ 
             $1, -2, -1, -1$ &$4, 9, 12, 15, 18, 21, 24, 27, \dots$ & $3M + 3$ \\ 
             $1, 0, 1, 1$ &$4, 8, 12, 16, 20, 24, 28, 32, \dots$ & $4M$ \\ 
             $2, -1, -2, 2$&$4, 9, 14, 18, 22, 26, 30, 34, \dots$ &$4M + 2$ \\ 
             $2, -2, -2, -1$&$4, 9, 16, 20, 24, 28, 32, 36, \dots$ & $4 M+4$ \\ 
             $3, -4, 2, 1$&$4, 10, 18, 25, 30, 35, 40, 45\dots$ & $5M+5$ \\ 
             $2, 1, 2, -1$&$4, 9, 15, 21, 27, 33, 39, 45, \dots$ & $6M - 3$ \\ 
             $5, -2, -5, -1$&$4, 9, 16, 22, 28, 34, 40, 46, \dots$ & $6M - 2$ \\
             $5, -3, -2, -8$&$4, 10, 17, 23, 29, 35, 41, 47 \dots$ & $6M - 1$ \\
             $3, -2, -3, 3$&$4, 9, 16, 23, 30, 36, 42, 48, \dots$ & $6M$ \\
             $2, -1, 0, 3$&$4, 10, 18, 26, 33, 39, 45, 51 \dots$ & $6M+3$ \\
             $3, -2, -3, -1$&$4, 9, 16, 25, 36, 42, 48, 54, \dots$ & $6M+6$ \\
             $1, 2, 2, -4$&$4, 9, 16, 24, 32, 40, 48, 56 \dots$ & $8 M-8$ \\
             $3, -2, -2, 4$&$4, 10, 18, 26, 34, 42, 50, 58 \dots$ & $8M - 6$ \\
             $4, -3, -2, 1$&$4, 10, 19, 28, 36, 44, 52, 60 \dots$ & $8M-4$ \\
             $0, -2, 4, -2$&$4, 10, 20, 32, 44, 54, 64, 72 \dots$ & $8M+8$ \\
             $1, 0, 3, -9$&$4, 9, 16, 25, 36, 48, 60, 72 \dots$ & $12M-24$ \\
             $0,0,5,-5$&$4, 10, 20, 33, 48, 62, 76, 88, 100, 110, 120, \dots$ & $10M+10$ \\ 
             $2, -1, 0, 1$&$4, 10, 19, 30, 42, 54, 66, 78 \dots$ & $12M-18$ \\ 
             $4, -3, -6, 9$&$4, 10, 20, 32, 44, 56, 68, 80 \dots$ & $12M-16$ \\ 
             $1, 1, 0, -6$&$4, 10, 20, 34, 50, 65, 78, 90, 102, \dots$ & $12M-6$ \\
             $4, -5, 2, 1$ &$4, 10, 20, \dots, 116, 130, 144, 156,\dots$ & $12M+12$ \\
             $3, 4, 2, -8$&$4, 10, 20, 33, 48, 64, 80, 96, 112, \dots$ & $16M-32$ \\
             $8, -7, -6, 1$&$4, 10, 20, 34, 51, 68, 84, 100, 116, \dots$ & $16M-28$ \\
             $1, 0, -1, -1$&$\boldsymbol{4, 9, 16, 25, 36, 49, 64, 81, \dots}$ & $M^2+2 M+1$ \\
             $2, 0, 0, -1$&$4, 10, 19, 31, 46, 64, 85, 109, \dots$ & $\frac{3 M^2}{2}+\frac{3 M}{2}+1$ \\
             $2, -1, 0, -1$&$4, 10, 20, 34, 52, 74, 100, 130, \dots$ & $2M^2 + 2$ \\
             $1, 1, 0, 1$&$4, 10, 20, 35, 56, 83, 116, 155, \dots$ & $3 M^2-6 M+11$ \\
             $2, 0, 0, -2$&$4, 10, 20, 35, 56, 84, 120, 164, \dots$ & $4 M^2-16 M+36$ \\
             $1, 2, 0, -3$ &$4, 10, 20, \dots, 220, 286, 364, 454, 556, \dots$ & $6 M^2-48 M+166$ \\
             $-3, 39, 47, 210$&$\boldsymbol{4, 10, 20, \dots, 220, 286, 364, 455, \dots}$ & $\frac{M^3}{6}+M^2+\frac{11 M}{6}+1$ \\
            \hline
        
             $0, 0, 0, 0, 1$ &$5, 5, 5, 5, 5, 5, 5, 5, \dots$ & $5$ \\ 
             $1, -1, 1, -1, 1$ &$5, 6, 6, 6, 6, 6, 6, 6, \dots$ & $6$ \\ 
             $1, 0, 0, -1, 1$ &$5, 8, 8, 8, 8, 8, 8, 8, \dots$ & $8$ \\ 
             $2, -2, 2, -2, 1$&$5, 9, 10, 10, 10, 10, 10, 10, \dots$ & $10$ \\ 
             $0, -1, 1, 0, 1$&$5, 10, 12, 12, 12, 12, 12, 12, \dots$ & $12$ \\ 
             $1, -1, -1, 1, -1$&$5, 9, 12, 15, 18, 21, 24, 27, \dots$ & $3M + 3$ \\
             $1, 0,0, 1, -1$&$\boldsymbol{5, 9, 13, 17, 21, 25, 29, 33, \dots}$ & $4M+1$ \\
             $1, 1, -1, 2, -2$&$5, 10, 14, 18, 22, 26, 30, 34, \dots$ & $4M+2$ \\
             $1, -1, -1, 0, 2$&$5, 11, 16, 20, 24, 28, 32, 36,\dots$ & $4M+4$ \\
             $1, 1, 1, 1, 2$&$5, 10, 15, 20, 25, 30, 35, 40, \dots$ & $5M$ \\
             $0, 1, 1, 0, -1$&$5, 11, 17, 23, 29, 35, 41, 47, \dots$ & $6M-1$\\
             $1, 2, -2, -2, 2$&$5, 11, 19, 27, 35, 43, 51, 59, \dots$ & $8M - 5$ \\
             $1, 1, -1, 1, -1$&$5, 11, 19, 28, 36, 44, 52, 60, \dots$ & $8M-4$ \\
             $2, -1, 1, 0, -1$&$5, 12, 20, 28, 36, 44, 52, 60, \dots$ & $8M-4$ \\
             $1, 1, 1, -1, -1$&$5, 12, 22, 34, 46, 58, 70, 82,\dots$ & $12M-14$ \\
             $0,0,1, -2, 2$&$5, 14, 29, 45, 58, 70, 82, 94, \dots$ & $12M-2$ \\
             $2, -1, -2, 1, -1$&$5, 15, 33, 57, 83, 108, 130, 150, \dots$ & $20M-10$ \\
             $1, 1, -1, -2, 2$&$5, 11, 19, 29, 41, 55, 71, 89, \dots$ & $M^2+3 M+1$ \\
             $1, 2, -1, 1, 2$&$\boldsymbol{5, 12, 22, 35, 51, 70, 92, 117, \dots}$ & $\frac{3 M^2}{2}+\frac{5 M}{2}+1$ \\
             $2, 0, 0, 2, -1$&$5, 13, 25, 41, 61, 85, 113, 145,\dots$ & $2 M^2+2 M+1$ \\
             $1, 1, 0,0, -1$&$5, 14, 29, 50, 77, 110, 149, 194, \dots$ & $3 M^2+2$\\
             $2, 0, -2, 0, 1$&$5, 14, 30, 54, 86, 126, 174, 230, \dots$ & $4 M^2-4 M+6$ \\
             $1, 0, 2, 0, 1$&$5, 15, 33, 60, 96, 141, 195, 258, \dots$ & $\frac{9 M^2}{2}-\frac{9 M}{2}+6$ \\
             $1, -2, -2, 2, 2$&$5, 15, 35, 67, 111, 165, 229, 301,\dots$ & $4 M^2+12 M-51$ \\
             $1, -1, 0, 0, -1$&$5, 14, 30, 55, 91, 139, 199, 271, \dots$ & $6 M^2-18 M+31$ \\
             $2, 1, -2, -1, -1$&$5, 15, 35, 67, 111, 167, 235, 315, \dots$ & $6 M^2-10 M+11$ \\
             $1, 1, -1, -1, -1$&$5, 15, 35, 68, 116, 180, 260, 356, \dots$ & $8 M^2-24 M+36$ \\
             $1, 0, 0, 0, 1$&$5, 15, 34, 65, 111, 174, 255, 354, \dots$ & $9 M^2-36 M+66$ \\
             $2, 1, 0, 0, -2$&$\boldsymbol{5, 14, 30, 55, 91, 140, 204, 285, \dots}$ & $\frac{M^3}{3}+\frac{3 M^2}{2}+\frac{13 M}{6}+1$ \\
             $1, 0, -1, 0, -1$&$5, 15, 35, 69, 121, 194, 290, 410, \dots$ & $12 M^2-60 M+122$ \\
             $1, 1, 0, 0, 2$&$5, 15, 34, 65, 111, 175, 260, 369, \dots$ & $\frac{M^3}{2}+\frac{3 M^2}{2}+2 M+1$ \\
             $1, 0, -1, 0, 1$&$5, 15, 35, 69, 121, 195, 295, 425, \dots$ & $\frac{2 M^3}{3}+M^2+\frac{7 M}{3}+1$ \\
             $1, -1, 2, 0, 1$&$5, 15, 35, 70, 126, 209, 325, 480, \dots$ & $M^3-\frac{3 M^2}{2}+\frac{17 M}{2}-4$ \\
             $2, 0, 0, 0, -1$&$5, 15, 35, 70, 126, 210, 330, 494, \dots$ & $\frac{4 M^3}{3}-6 M^2+\frac{86 M}{3}-34$\\
             $2, 0, 0, -1, -1$&$5, 15, 35, 70, \dots, 1001, 1365, 1819, \dots$ & $ 2 M^3-21 M^2+143 M-329$ \\
             $3, 23, -51, -94, 120$&$\boldsymbol{5, 15, 35, 70, \dots, 1001, 1365, 1820, \dots}$ &$\frac{M^4}{24}+\frac{5 M^3}{12}+\frac{35M^2}{24}+\frac{25 M}{12}+1$\\
            \hline
    \end{longtable}
\end{document}